\documentclass[a4paper,12pt,reqno]{amsart}
\usepackage[utf8]{inputenc}
\usepackage{amsmath}
\usepackage{amsthm,amssymb,amsfonts}
\usepackage{graphicx}
\usepackage{hyperref}
\usepackage{url}

\usepackage{geometry}
\usepackage{array,bm}
\usepackage{color}
\usepackage{verbatim} 

\newtheorem{theorem}{Theorem}

\newtheorem{proposition}[theorem]{Proposition}
\newtheorem{remark}[theorem]{Remark}

\newtheorem{lemma}[theorem]{Lemma}

\renewcommand\Im{\operatorname{Im}}
\renewcommand\Re{\operatorname{Re}}

\makeatletter
\@namedef{subjclassname@2020}{\textup{2020} Mathematics Subject Classification}
\makeatother

\begin{document}

\title[On a pointwise inequality for even Legendre polynomials]{On a pointwise inequality for even Legendre polynomials in high dimensional spheres}

\dedicatory{Dedicated to Tohru Ozawa sensei, with admiration and gratitude}

\author{Shirong Chen}
\address{School of Mathematics and Computer Sciences, Gannan Normal University, Ganzhou 341000, People's Republic of China}
\email{somxiaorong@163.com}

\author{Yi C. Huang} 
\address{Yunnan Key Laboratory of Modern Analytical Mathematics and Applications, Yunnan Normal University, Kunming 650500, People's Republic of China}
\address{School of Mathematical Sciences, Nanjing Normal University, Nanjing 210023, People's Republic of China}
\email{Yi.Huang.Analysis@gmail.com}
\urladdr{https://orcid.org/0000-0002-1297-7674}

\author{Jian-Yang Zhang}
\address{School of Mathematical Sciences, Nanjing Normal University, Nanjing 210023, People's Republic of China}
\email{3203927393@qq.com}

\date{\today}
\thanks{Research of YCH is partially supported by the National NSF grant of China (no. 11801274),
the JSPS Invitational Fellowship for Research in Japan (no. S24040),
and the Open Project from Yunnan Normal University (no. YNNUMA2403).
YCH thanks in particular Mikhail Tyaglov \& Yuzhe Zhu for helpful communications on Orthogonal Polynomials \& Kinetic Equations.}

\maketitle

\begin{abstract}
	We present a pointwise inequality for adjacent even Legendre polynomials in high dimensional spheres featuring the effect of spectral gaps.
	This improves a recent result of Imbert, Silvestre and Villani that is crucially used in their study of the Fisher information for the Boltzmann equation.
\end{abstract}

\section{Introduction}

Recently, in proving the monotonicity of the Fisher information for the Boltzmann equation,
Imbert, Silvestre and Villani established certain \textit{integro-differential} log-Sobolev inequalities and Poincar\'e inequalities for the Boltzmann operator
$$\mathcal{B}_b(f)(\sigma)=\int_{\mathbb{S}^{d-1}}(f(\sigma')-f(\sigma))b(\sigma'\cdot\sigma)d\sigma'.$$
In particular, the study of their Poincar\'e inequalities is reduced to the following interesting pointwise inequality for even Legendre polynomials, see \cite[Prop. 7.6]{Villani24}. 
\begin{proposition}[Imbert, Silvestre and Villani, 2024] \label{prop:Imb}
	For the Legendre polynomials of order $\ell$ normalized so that $P_\ell(1)=1$, we have for all $x\in[-1,1]$ and $\ell\ge1$,
	\begin{equation}\label{Legendre estimate}
		1-P_{2\ell}(x)\le\frac{\lambda_{2\ell}}{\lambda_2}(1-P_2(x)).
	\end{equation}
	Here $\lambda_\ell=\ell(\ell+d-2)$ is the eigenvalue of the (minus) spherical Laplacian $-\Delta_{\mathbb{S}^{d-1}}$ 
	corresponding to spherical harmonics of order $\ell$, and the unique axially symmetric spherical harmonic $Y_\ell$ of order $\ell$
	such that $Y_\ell(e_1)=1$ is given by $Y_\ell(\sigma):=P_\ell(e_1\cdot\sigma)$.
\end{proposition}

More precisely, \eqref{Legendre estimate} is used in \cite{Villani24} to prove the following eigenvalue comparison
	\begin{equation}\label{eigencomp}
\frac{\widetilde{\lambda}_{2\ell}}{{\lambda}_{2\ell}}\leq \frac{\widetilde{\lambda}_{2}}{{\lambda}_{2}},\quad\ell\geq1,
	\end{equation}
where $\widetilde\lambda_\ell$ is the eigenvalue of the (minus) Boltzmann operator $-\mathcal{B}_b$ corresponding to spherical harmonics of order $\ell$.
From \eqref{eigencomp} one obtains immediately the Poincar\'e inequalities: for some dimensional and $b$-dependent constant $C_P(d,b)$, we have
$$(-\Delta_{\mathbb{S}^{d-1}})\ge C_P(d,b)(-\mathcal{B}_b)$$ 
in the sense of quadratic form on functions which are even on the sphere in the sense $f(\sigma)=f(-\sigma)$. Moreover, $C_P(d,b)$ can be explicitly evaluated: for any $e\in\mathbb{S}^{d-1}$
$$C_P(d,b)=\frac12\frac{{\lambda}_{2}}{\widetilde{\lambda}_{2}}=\frac{d-1}{\int_{\mathbb{S}^{d-1}}(1-(e\cdot\sigma)^2)b(e\cdot\sigma)d\sigma}.$$

Here we present an improvement of Proposition \ref{prop:Imb} featuring the spectral gaps .

\begin{proposition} \label{p:LegPolGap}
	For all $x\in [-1,1]$ and $\ell\ge1$,
	\begin{equation}\label{improved Legendre estimate}
		P_{2\ell}(x)-P_{2(\ell+1)}(x)\le\frac{\lambda_{2(\ell+1)}-\lambda_{2\ell}}{\lambda_2}(1-P_2(x)).
	\end{equation}
\end{proposition}

\begin{remark}
Apparently, a telescoping summation of \eqref{improved Legendre estimate} leads to \eqref{Legendre estimate}.
\end{remark}

\begin{remark}
It would be interesting to find a purely spectral proof of \eqref{Legendre estimate} or \eqref{improved Legendre estimate}.
\end{remark}

\section{Proof of Proposition \ref{p:LegPolGap}}

Our proof is based on the following crucial observation. 
When $d=2$, $P_\ell(x)$ corresponds to $\cos(\ell\Theta)$ if $x=\cos\Theta$ with $\Theta\in[0,\pi]$,
so in this case \eqref{improved Legendre estimate} becomes 
\begin{equation}\label{unit case}
\cos(2\ell\Theta)-\cos(2(\ell+1)\Theta)\le(2\ell+1)(1-\cos(2\Theta)).
	\end{equation}
By the \underline{sum-to-product formula} for trigonometric functions, \eqref{unit case} is equivalent to
\[\sin((2\ell+1)\Theta)\le(2\ell+1)\sin\Theta.\]
Therefore, it is enough to prove the following elementary estimate
\[|\sin((2\ell+1)\theta)|\le(2\ell+1)|\sin\theta|,\quad \theta\in[-\pi,\pi],\]
which can be verified simply by induction on $\ell$, see for example \cite{Villani24}.

For higher dimensions, we need the following basic tool, see \cite[Thm. 4.26]{Efthimoubook2014}.
\begin{lemma}[Integral representation formula]\label{integral formula for LP}
	Assume $d\ge3$. We have
	\[P_\ell(x)=\frac{\omega_{d-3}}{\omega_{d-2}}\int_{-1}^1\left(x+is\sqrt{1-x^2}\right)^{\ell}(1-s^2)^{\frac{d-4}{2}}\,ds.\]
	Here, for $p\geq0$, the dimensional constant $\omega_p$ is $\frac{2\pi^{\frac{p+1}{2}}}{\Gamma(\frac{p+1}{2})}$.
\end{lemma}
Note that by the \underline{double-angle formula}, \eqref{unit case} is also equivalent to
\begin{equation}\label{unit case'}
\sin^2((\ell+1)\theta)-\sin^2(\ell\theta)\le(2\ell+1)\sin^2\theta.
	\end{equation}
	We also need the following technical result which is somehow inspired by \eqref{unit case'}.
\begin{lemma}\label{genlized version of unit case}
	The following inequality holds for all $\ell\ge1$ and $x,s\in[-1,1]$,
	\[\left(\Im \left(x+is\sqrt{1-x^2}\right)^{\ell+1}\right)^2-\left(\Im \left(x+is\sqrt{1-x^2}\right)^{\ell}\right)^2\le(2\ell+1)s^2(1-x^2).\]
\end{lemma}
\begin{proof}
	Fix $x,s\in[-1,1]$. Note that $$\left|x+is\sqrt{1-x^2}\right|^2=x^2+s^2(1-x^2)\le1.$$
	Let $r= \left|x+is\sqrt{1-x^2}\right|^{-1}\geq1$. We set 
	$$\sin\theta=rs\sqrt{1-x^2},\quad \theta\in[-\pi,\pi]$$ and 
$$A=r^{2\ell}\left(\left(\Im \left(x+is\sqrt{1-x^2}\right)^{\ell+1}\right)^2-\left(\Im \left(x+is\sqrt{1-x^2}\right)^{\ell}\right)^2\right).$$
	Since $r\ge1$, we have
	\begin{equation}\label{trans to unit}
			A\le r^{2(\ell+1)}\left(\Im \left(x+is\sqrt{1-x^2}\right)^{\ell+1}\right)^2-r^{2\ell}\left(\Im \left(x+is\sqrt{1-x^2}\right)^{\ell}\right)^2.
	\end{equation}
	By the construction of $r$ and $\theta$, the RHS of $\eqref{trans to unit}$ actually equals to $\sin^2((\ell+1)\theta)-\sin^2(\ell\theta)$ which is less than $(2\ell+1)\sin^2\theta=(2\ell+1)s^2r^2(1-x^2)$ by \eqref{unit case'}.
Thus we get
	$$\begin{aligned}
	&r^{2(\ell-1)}\left(\left(\Im \left(x+is\sqrt{1-x^2}\right)^{\ell+1}\right)^2-\left(\Im \left(x+is\sqrt{1-x^2}\right)^{\ell}\right)^2\right)\\
	&\qquad\le (2\ell+1)s^2(1-x^2).\end{aligned}$$
	Since $r\ge1$ and $\ell\geq1$, we can drop above factor $r^{2(\ell-1)}$ and finish the proof.
\end{proof}
 We now continue the proof of Proposition \ref{improved Legendre estimate}.
	By Lemma \ref{integral formula for LP}, we need estimate
	\[\frac{\omega_{d-3}}{\omega_{d-2}}\int_{-1}^1\left(\left(x+is\sqrt{1-x^2}\right)^{2\ell}-\left(x+is\sqrt{1-x^2}\right)^{2(\ell+1)}\right)(1-s^2)^{\frac{d-4}{2}}\,ds.\]
	It is easy to see that this integral is actually a real number, so we only need estimate the real part of the difference-form integrand. 
	For any complex number $z=a+ib$, we have $\Re  z^2=|z|^2-2b^2$. 
	Applying this fact to 
	$$F(s,x)=\Re  \left(\left(x+is\sqrt{1-x^2}\right)^{2\ell}-\left(x+is\sqrt{1-x^2}\right)^{2(\ell+1)}\right)$$
    and using Lemma \ref{genlized version of unit case}, we get
    \[F(s,x)\le(4\ell+2)s^2(1-x^2)+(x^2+s^2(1-x^2))^{\ell}(1-x^2-s^2(1-x^2)).\]
    As $x^2+s^2(1-x^2)\le1$, this gives the following simpler estimation 
    $$F(s,x)\le(4\ell+2)s^2(1-x^2)+(1-x^2)(1-s^2).$$
    Now, integrating above inequality, we get
\begin{align*}
	P_{2\ell}&(x)-P_{2(\ell+1)}(x)\\
	&=\frac{\omega_{d-3}}{\omega_{d-2}}\int_{-1}^1F(s,x)(1-s^2)^{\frac{d-4}{2}}\,ds\\
	&\le\frac{\omega_{d-3}}{\omega_{d-2}}(1-x^2)\int_{-1}^1\left(4\ell+2-(4\ell+1)(1-s^2)\right)(1-s^2)^{\frac{d-4}{2}}\,ds\\
	&=\frac{\omega_{d-3}}{\omega_{d-2}}(1-x^2)\left((4\ell+2)\sqrt{\pi}\frac{\Gamma\left(\frac d2-1\right)}{\Gamma\left(\frac {d+1}{2}-1\right)}-(4\ell+1)\sqrt{\pi}\frac{\Gamma\left(\frac d2\right)}{\Gamma\left(\frac {d+1}{2}\right)}\right)\\
	&=\left(4\ell+2-(4\ell+1)\frac{d-2}{d-1}\right)(1-x^2)\\
	&=\frac{4\ell+d}{d-1}(1-x^2).
\end{align*}
Since $P_2(x)=x^2+\frac{1-x^2}{d-1}$ (computed by the rotation symmetry of $Y_2$) and
$$\frac{\lambda_{2(\ell+1)}-\lambda_{2\ell}}{\lambda_2}(1-P_2(x))=\frac{4\ell+d}{d-1}(1-x^2),$$
this proves the proposition.

\bigskip

\section*{\textbf{Compliance with ethical standards}}

\bigskip

\textbf{Conflict of interest} The authors have no known competing financial interests
or personal relationships that could have appeared to influence this reported work.

\bigskip

\textbf{Availability of data and material} Not applicable.

\end{document}